\documentclass{article}
\usepackage[utf8]{inputenc}
\usepackage{amsfonts}
\usepackage{amsmath}
\usepackage{amssymb}
\usepackage{amsthm}
\usepackage{tikz}
\newtheorem{theorem}{Theorem}[section]

\newtheorem{lemma}[theorem]{Lemma}
\newtheorem{definition}{Definition}[section]
\newcommand{\zero}{\alpha}
\newcommand{\one}{\beta}
\newcommand{\cylx}[1]{[#1]_X}
\newcommand{\cylz}[1]{[#1]_Z}

\usepackage{graphicx}
\usepackage [english]{babel}
\usepackage [autostyle, english = american]{csquotes}
\MakeOuterQuote{"}
\title{\textbf{Failure of the Gibbs inequality for continuous potentials }}
\author{Arantha Ranu}
\date{\today}

\begin{document}

\maketitle

\begin{abstract}
    It is well known that the Gibbs inequality, which says that the Gibbs ratio is bounded above and below by positive constants, holds for the unique equilibrium states of Hölder continuous potentials on shift spaces, but it can fail for continuous potentials. In this article, we study the validity of a weaker form of the Gibbs inequality in this broader setting.

\end{abstract}

\section{{Introduction}}\label{intro}
In \cite{zbMATH03498229} Walters proved the variational principle which states that if $X$ is a compact metric space and  $T:X\to X$ is a continuous map on $X$ and $\psi$ is a continuous fucntion from $X$ to $\mathbb R$, then the pressure $P(\psi)$ of the the potential is given by,
                \begin{align*}
                    P(\psi)=\sup_{\eta} \left(h_{\eta}(\sigma)+\int \psi\, d\eta\right)
                \end{align*} where the supremum is taken over all $T$- invariant measures $\eta$ and $h_\eta(T)$ is the entropy of $T$ with respect to $\eta$. In \cite{bowen1974some} Bowen showed that for continuous potentials $\psi$ on shift spaces  the supremum is a maximum. The measures at which the maximum is attained are called \textit{equilibrium states}.\\

                \noindent In \cite{bowen1974some} it is further proved that for Holder continuous potential $\psi$, there exists  a  unique equilibrium state $\mu_{\psi}$.  Also, there exists $c>0$ such that,
                \begin{align}\label{bowen}
    \frac{1}{c}<\frac{\mu_{\psi}(W_n(z))}{e^{-nP(\psi)+S_n\psi(z)}}<c
\end{align} for all $n$- length word $W_n$  in the language of the finite shift and all $z\in W_n$. The unique equilibrium state is called a \textit{Gibbs state} and the property (\ref{bowen}) satisfied by it is called the \textit{Gibbs property} \\
\noindent This is a very important and useful property and has various uses. One of them is to find the Hausdorff dimensions of certain dynamical objects, an example of which can be found in the paper \cite{zbMATH03685851}. Now a natural question is to ask if such upper and lower bounds still exist if we weaken continuity condition on the potentials. \\

\noindent In \cite{zbMATH05229033} Huyi Hu mentions the \enquote{weak Gibbs measures} for potentials $\psi$ that are Holder at all but one point. The equilibrium states of these potentials satisfy,
\begin{align}\label{hu}
    \frac{1}{p(z,n)}<\frac{\mu_{\psi}(W_n(z))}{e^{-nP(\psi)+S_n\psi(z)}}<p(z,n)
\end{align}
where $\log_{n\to \infty}\frac{1}{n} \log p(z,n)=0$ for all $z$ except for the point at which the function is non Holder.\\

\noindent In \cite{zbMATH03552862} Hofbauer gave an example of a continuous potential on the one-sided shift on two symbols that has two equilibrium states. One of the equilibrium states for this example is the Dirac measure at the point of all ones. This shows that there there is no suitable positive lower bound for every point in the shift space for the class of continuous potentials. Now the question arises, does there exist a suitable upper bound for the class of continuous potentials?\\
\noindent For a shift space $(Z,\sigma)$ and a Holder continuous potential $\psi$ we have the distortion property i.e. for all $n\in \mathbb N$ there exists $C>0$ such that 
\begin{align*}
        \frac{1}{C}<\frac{e^{S_n\psi(z)}}{e^{S_n\psi(z)}}<C
        \end{align*}
        for all $z,z'$ in the same $n-$ cylinder. \\
\noindent A weak version of the distortion property holds for the class of continuous potentials as follows, for a continuous potential $\psi$ there exists $C>0$ such that for all $n\in \mathbb N$. 
        \begin{align*}
        Ce^{-n\epsilon}<\frac{e^{S_n\psi(z)}}{e^{S_n\psi(z')}}<Ce^{n\epsilon}
            \end{align*}
            for all $z,z'$ in the same $n-$ cylinder.

\noindent This weaker version of distortion motivates us to define the "very weak Gibbs inequality" and the above question about a suitable upper bound for the Gibbs ratio of continuous potentials will be answered by proving the following theorems. 
\begin{definition}
    Let $\psi$ be a continuous potential on a shift space $(Z,\sigma)$ that has an equilibrium state $\mu$. Let $z\in Z$ and for $n\in \mathbb N$ let $[{W_n}]$ be the $n$-cylinder containing $z$. We say that $\psi$ satisfies the "very weak Gibbs inequality" at $z$ if for $\epsilon>0$ there exists $C(z,\epsilon)$ such that the following is true,
    \begin{align}\label{-1}
    \frac{1}{C(z,\epsilon)e^{n\epsilon}}<\frac{\mu[W_n]}{e^{-nP(\psi)+S_n\psi(z)}}<C(z,\epsilon)e^{n\epsilon}
\end{align}
\end{definition}

\begin{theorem}\label{thm1}
For a continuous potential $\psi$  on a shift space $(Z,\sigma)$ that has an ergodic equilibrium state $\mu$. The "very weak Gibbs inequality" is satisfied for $\mu$ a.e. $z$.

\end{theorem}

\noindent As we mentioned earlier the example in  \cite{zbMATH03552862} has the Dirac measure as one of its equilibrium states which means the lower bound of the "very weak Gibbs inequality" is not true for every point in the shift space. The following Theorem states that $\frac{\mu[W_n]}{e^{-NP(\psi)+S_N\psi(o)}}$ grows exponentially which shows that the  upper bound is also not true for the point $o$ of all $\one'$s.

\begin{theorem}\label{thm2}
There exists a continuous potential $\psi$  on a full shift that has a unique equilibrium state $\mu$ such that for all $n\geq 5$ we have
\begin{align*}
    \frac{\mu[W_n]}{e^{-NP(\psi)+S_N\psi(o)}}>\left(\frac{e}{2}\right)^n
\end{align*}
where $o$ is a fixed point of $\sigma$ and $W_n$ is the $n$- word containing $o$.
\end{theorem}

\noindent Theorem \ref{thm2} is proved in section \ref{pot} and Theorem \ref{thm1} is proved in section \ref{a.e.}.               
\section{Preliminaries}\label{prelim}
In this section, we will recall some properties of odometers and shift spaces.\\

\noindent \underline{\textbf{Properties of odometer maps:-}}\\

\noindent Let $X=\{0,1\}^{\mathbb N_0}$ where an element $x\in X$ is of the form $x=(x_i)_{i\in \mathbb N_0}$. 
\begin{definition}
    An $n$- \textbf{cylinder} in $X$ denoted by $\cylx{(a_0\ldots a_{n-1})}$ is defined as $\cylx{a_0\ldots a_{n-1}}=\{x\in X: x_0=a_0,\ldots x_{n-1}=a_{n-1}\}$, for $n\in \mathbb N$ and $a_0,\ldots,a_{n-1}\in \{0,1\}$. 
\end{definition}

\noindent \begin{definition}
    Let $x\in \cylx{a_0\ldots a_{n-1}}\subset  X$. For $k\in \mathbb N$, we denote $\mathcal{N}_k(x)=a_{k-1}a_{n-2}\ldots a_0$ which is the \textbf{$k$- digit binary sequence} that $x$ ends with.
\end{definition}

\begin{definition}
    Define $\mathcal{DN}_k(x)=2^{k-1}a_{k-1}+\ldots+2a_1+a_0$ which is the \textbf{number corresponding to the binary sequence} $\mathcal{N}_k(x)$.
\end{definition}
\begin{definition}
    The map $T:X\to X$ called the \textbf{odometer map} is defined as,
    $T(x)_k=$ 
    $\begin{cases} 
        0 & \text{if, }\mathcal{N}_k(x)=11\ldots11\\
        1 & \text{if, }\mathcal{N}_k(x)=01\ldots 11\\
        x_k & \text{otherwise}
    \end{cases}$
\end{definition}
\noindent {Here are a few examples to show how the map works.}
\begin{align*}
    &T(\ldots011110)=\ldots011111\\
    &T(\ldots011111)=\ldots100000
\end{align*}

\noindent The odometer map can be seen as \enquote{adding $1$ with carry in binary}. We conclude the following from this:-
\begin{itemize}
    \item $\mathcal{DN}_k(T(x))=\mathcal{DN}_k(x)+1\mod 2^k$.
    \item The set, $\{\mathcal{DN}_k(x), \mathcal{DN}_k(T(x)), \ldots, \mathcal{DN}_k(T^{2^k-1}(x))\} $ is equal to $ \{0,..,2^k-1\}$. 
    \item For the $k$- cylinder set $\cylx{a_0a_1\ldots a_{k-1}}$, for all $x\in X$, there exists $i\in \{0,\ldots,2^{k}-1\}$ such that $T^i(x)\in \cylx{a_0a_1\ldots a_{k-1}}$.
\end{itemize}

\noindent The inverse of the odometer map is defined as,\\

$T^{-1}(x)_k=$ 
    $\begin{cases} 
        1 & \text{if, }\mathcal{N}_k(x)=00\ldots 00\\
        0 & \text{if, }\mathcal{N}_k(x)=10\ldots00\\
        x_k & \text{otherwise}
    \end{cases}$\\
\begin{definition}
    For $x\in X$, let $\kappa(x)$ denote the number of zeros that $x$ ends with. It is defined as follows:
    \begin{align*}
        \kappa(x)=\begin{cases}
            n & \text{if } x\in \cylx{10^n}\\
            \infty & \text{if } x=\overline{0} \text{ is the all zero element.}
        \end{cases}
    \end{align*}
\end{definition}
\begin{definition}
Define the metric $d_X(x,x')=2^{-\min\{i:x_i\neq x_i'\}}$ for $x, x'\in X$.
\end{definition}
\noindent \textbf{Note:-}  The Borel $\sigma$- algebra on $X$ is generated by the cylinder sets and they are clopen sets.
\begin{definition}
    We define a \textbf{measure} $\nu$ on the Borel $\sigma$- algebra of $X$ by specifying the measure of cylinder sets as,
    
    \begin{align*}
        \nu(\cylx{a_0\ldots a_{n-1}})=\frac{1}{2^{n}}
    \end{align*} for each $n$ and each cylinder $\cylx{a_0,\ldots a_{n-1}}$.
\end{definition}
\noindent It is a well known fact that the odometer map $T$ is uniquely ergodic and $\nu$ is the unique invariant ergodic measure\cite{MR2180227}. Also from \cite{zbMATH01542660} (pg. 101 Theorem 4.25) the entropy of $T$ w.r.t $\nu$ is 0 .\\

\noindent\underline{\textbf{Properties of shift spaces:-}}\\

\noindent Let $Z=\{\zero, \one\}^\mathbb Z$ be the full shift on the alphabet $\{\zero,\one\}$. An element $z\in Z$ is denoted as : $\ldots z_{-2}z_{-1}.z_0z_1z_2\ldots$ where the \enquote{.} separates positive and negative indices.

   \begin{definition}
   The cylinder sets in $Z$ are of the form $\cylz{\omega_0\ldots\omega_n}=\{z: z_0=\omega_0,\ldots,z_n=\omega_n\}$ or $\cylz{\omega_{-i}\ldots\omega_{-1}.\omega_0\ldots\omega_j}=\{z:z_{-i}=\omega_{-i},\ldots,z_{-1}=\omega_{-1},z_0=\omega_0,\ldots z_{j}=\omega_j\}$ where the \enquote{.} separates the positive and the negative indices.
\end{definition} 

\begin{definition}
    The \textbf{shift map} $\sigma: Z\to Z$, is defined as,
\begin{align*}
    \sigma(\ldots z_{-2}z_{-1}.z_0z_1z_2\ldots)_k=\ldots z_{-2}z_{-1}z_0.z_1z_2\ldots \hspace{2mm} .
\end{align*}
\end{definition} 
\noindent  The shift map is invertible and the inverse is defined as $\sigma^{-1}: Z\to Z$ as $\sigma^{-1}(\ldots z_{-2}z_{-1}.z_0z_1z_2\ldots)=\ldots z_{-2}.z_{-1}z_0z_1z_2\ldots$.

\begin{definition}
    The $i,j$- \textbf{segment} of an element $z\in Z$ denoted by $(z)_i^j$ is defined as the sequence $(z)_i^j=z_i\ldots z_j$ for all $i,j\in \mathbb Z$.
\end{definition}

\noindent \textbf{Remark:-} Note that $(z)_{i}^j=(\sigma^i(z))_0^{j-i}$.

\begin{definition}
    Define the $i$-\textbf{ right segment}, $(z)_i^\infty=z_iz_{i+1}\ldots$, i.e. the right part of the sequence $z$ starting at $i$ for $i\geq 0$. Similarly define the $j$-\textbf{ left segment} as $(z)_{-\infty}^ j=\ldots z_{j-1}z_j$ for $j<0$. 
\end{definition}
\noindent \textbf{Notation:-} The $i$-right segment and $j$-left segment consisting only of $\one'$s are both denoted by $\overline{\beta}$. The point of all $\one'$s is denoted by $o$. That means $(o)_{0}^\infty=\overline{\one}$ and $o_{-\infty}^{-1}=\overline{\one}$.\\

\noindent We will use the metric given by, $d_Z(z,z')=2^{-\min\{|i|:z_i=z'_i\}}$ for $z,z'\in Z$.

\begin{definition}
    A nonempty, closed and shift invariant subspace $Y$ of $Z$ equipped with the shift map $\sigma$ is called a \textbf{subshift}.
\end{definition} 

\begin{definition}
    For $y\in Y$, $(y)_0^{k-1}$ is called a \textbf{$k$- letter word} in $Y$ and we denote $\mathcal{L}_k(Y)$ to be the set of all $k$- letter words in $Y$. 
\end{definition}

\noindent \textbf{Remark:-} For a $k$-cylinder, if $\cylz{\omega_0\ldots\omega_{k-1}}\cap Y\neq \phi$ then $(\omega)_0^{k-1}\in \mathcal{L}_k(Y)$.\\

\noindent \textbf{\underline{Some important Theorems:-}}\\

\noindent Now we will recall the Shannon-McMillan-Breiman Theorem and the Birkhoff Theorem using which we will prove two lemmas. For the following lemmas we consider the space $(X,T)$ and $\eta$ be a $T$-invariant ergodic measure on $X$.

\begin{theorem}[Shannon-McMillan-Breiman]\label{smb}
Let $\mathcal{P}$ be a countable or finite partition of $X$. Let $\mathcal{P}_n=\bigvee_{i=0}^{n-1} \sigma^{-i}\mathcal{P}$ and for $x\in X$ and let $\mathcal{P}_n(x)$ be the element of $\mathcal{P}_n$ to which $x$ belongs. For $\mu$ a.e. $x$ the following is true,
\begin{equation}
    \lim_{n\to \infty} \left[-\frac{\log(\eta(\mathcal{P}_n(x)))}{n}\right]=h_\eta(\sigma,\mathcal{P})
\end{equation}

\noindent where $h_\eta(\sigma,\mathcal{P})$ is the entropy of $X$ w.r.t the partition $\mathcal{P}$. 
\end{theorem}

\begin{theorem}[Birkhoff Theorem]\label{b}
    Let $f:X\to \mathbb R$ belong to $L^1(\eta)$.  For $\eta$- a.e. $x$ the following is true,
    \begin{align}
        \lim_{n\to \infty }\frac{S_nf(x)}{n}=\int f\,d\eta
    \end{align}
    where ${S_nf(x)}=\sum_{i=0}^{n-1}f(T^i(x))$.
\end{theorem}

\begin{theorem}[weak Birkhoff Theorem]\label{wb}
                     Let $\epsilon,\delta>0$ and $f\in L^1(\eta)$ there exists $N\in \mathbb N$ such that such that for all $n\geq N$
                    \[\eta\left(\{x:\left|\frac{S_nf(x)}{n}-\int f\,d\eta\right|<\epsilon\}\right)>1-\delta\]
                \end{theorem}
                \begin{proof}
                Let $G_j=\{x:\left|\frac{S_jf(x)}{j}-\int f\,d\eta\right|<\epsilon\}$ and
                    let $H_n=\bigcap_{j=n}^\infty G_j$. The sequence of sets $(H_n)_{n\in \mathbb N}$ is increasing. So we have
                    \begin{align*}
                        \eta\left(\bigcup_{n=1}^\infty H_n\right)=\lim_{n\to \infty}\eta\left(H_n\right)
                    \end{align*}
                    
                    \noindent By the Birkhoff Theorem we know that $\eta\left(\bigcup_{n=1}^\infty H_n\right)=1$. Hence there exists $N\in \mathbb N$ such that for all $n\geq N$ we have $\eta\left(H_n\right)>1-\delta$. Since $H_n\subseteq G_n$  so we have $\eta(G_n)>1-\delta$ for all $n\geq N$, i.e. 
                    \[\eta\left(\{x:\left|\frac{S_nf(x)}{n}-\int f\,d\eta\right|<\epsilon\}\right)>1-\delta\]
                \end{proof}
                \begin{theorem}[weak Shannon-McMillan-Breiman Theorem]\label{wsmb}
                     Let $\epsilon,\delta>0$ and $\{C_1,\ldots,C_l\}$ be a collection of $n$-cylinders that satisfy $\eta(C_j)\in [e^{-n(h_\eta+\epsilon)}, e^{-n(h_\eta-\epsilon)}]$, there exists $N\in \mathbb N$ such that, for all $n\geq N$ we have
                    \[\sum_{j=1}^l\eta(C_j)>1-\delta\]
                \end{theorem}
                \begin{proof}
                The proof is same as Theorem \ref{wb}
                    
                \end{proof}

\noindent In sections \ref{const} and \ref{mu} we will construct a subshift $Y$ and a measure $\mu$ such that $\mu(Y)=1$. Finally in section \ref{pot} we will construct a potential $\psi$ on $Z$ that has $\mu$ as its unique equilibrium state. We then use and $\psi$ and $\mu$ to prove Theorem \ref{thm2}.

\section{Construction and properties of the subshift $(Y,\sigma)$}\label{const}

\noindent In this section we will construct a subspace $(Y,\sigma)$ that has the following properties:
\begin{itemize}
    
    \item  Every word in $\mathcal{L}_{2^l}(Y)$ has at least $\frac{l}{2}$ consecutive $\beta$'s. In fact the words have at least $l-1$ $\one'$s.
    \item There are only two ergodic measures on $(Y,\sigma)$.
\end{itemize} 

\noindent\underline{\textbf{Construction:-}}\\

\noindent Let $A=\bigcup_{i=5}^{\infty} \bigcup_{j=0}^{i-1} T^j(\cylx{(0^i)})\subset X$. Define $\phi: X\to \{\zero,\one\}$ as follows:
\[
\phi(x) =
\begin{cases}
\one & \text{if } x\in A \\
\zero & \text{if } x\in A^c
\end{cases}
\]
Using $\phi$ we define a map $\Phi: X\to Z$ as,
\begin{align*}
    \Phi(x)_n=\phi(T^n(x))
\end{align*} for all $n\in \mathbb Z$.\\

\noindent Note that $\Phi$ satisfies $\sigma(\Phi(x))=\Phi(T(x))$ which can be seen as follows,
\begin{align}\label{blah}
    \sigma(\Phi(x))_n=\Phi(x)_{n+1}=\Phi(T^{n+1}(x))=\Phi(T^n(T(x)))=\Phi(T(x))_n
\end{align}
for all $n\in \mathbb Z$.\\

\noindent We will now consider the subshift  $Y=\overline{{\Phi(X)}}\subseteq Z$.\\

\noindent \underline{\textbf{Properties:-}}\\

\noindent Now we will understand a few important properties of the elements in $Y$. Lemma \ref{4.1} helps us understand the occurrence of $\one$'s in the the elements of $Y$ which is used later in section \ref{mu}.\\
\noindent Theorem $\ref{Y}$ is the most important result in this section which explicitly shows the types of elements that belong to $Y$. 


\begin{lemma}\label{4.1}
    Let $l\geq 5$ and $y\in Y$ be such that for all $i>0$, $(y)_i^\infty\neq \overline{\beta}$. Then, there exists $K\in \mathbb N$ such that, $\sigma^K(y)\in \cylz{\zero.\one^l}$.
\end{lemma}

\begin{proof}
     $y$ satisfies $(y)_i^\infty\neq \overline{\beta}$ so there exists $q>0$ for which $y_q=\zero$.\\
    \noindent Let, $\Tilde{y}=\sigma^q(y)$. That means $\Tilde{y}\in \cylz{\zero}$. Since $Y=\overline{\Phi(X)}$ there exists a sequence $(x^n)$ in $X$ such that, $\Phi(x^n)\to \Tilde{ y}$. So there exists $N\in \mathbb N$ such that $\Phi(x^n)\in \cylz{\zero}$  for all $n\geq N$, i.e. $\Phi(x^n)_0=\zero$. \\
    \noindent By the properties of odometers for each $n$ there exists $0\leq k_n<2^l$ such that $T^{k_n}(x^n) \in \cylx{0^l}$.\\
    \noindent For $n\geq N$, $\Phi(x^n)_0=\zero$ so $\Phi(x^n)\notin \cylz{\one^l}$. This means $x^n\notin\cylx{0^l}$, i.e. $k_n>0$.\\
    \noindent So we get an infinite sequence $(k_n)$ in the range $ \{1,\ldots 2^l-1\}$.\\
    \noindent By the Pigeon Hole Principle there exists a sequence $(n_i)$ such that, $(k_{n_i})=p$ for some $0<p<2^l$.\\
    \noindent So for the subsequence $(x^{n_i})$ 
    we have, $T^p(x^{n_i})\in\cylx{0^l}$. which means $\Phi(T^p(x^{n_i}))\in \cylz{\one^l}$.\\
    \noindent From (\ref{blah}) we have, $\sigma^p(\Phi(x^{n_i}))=\Phi(T^p(x^{n_i}))\in \cylz{\one^l} $. As $\{\Phi(x^{n_i})\}\to \Tilde{y}$, so $\sigma^p(\Phi(x^{n_i})\})\to \sigma^p(\Tilde{y})$. Therefore $\sigma^p(\Tilde{y})\in \cylz{\one^l}$, i.e. $\sigma^{p+q}(y)\in \cylz{\one^l}$ because $\Tilde{y}=\sigma^q(y)$. Let $K=\max\{i<p+q: y_i=\zero\}+1$. Note that we have $K-1\geq q>0$ because $y_q=\zero$. Thus, we have $\sigma^K(x)\in \cylz{\zero.\one}$.
    \noindent
\end{proof}
\begin{lemma}\label{4.2}
    Let $y\in Y$ be such that for all $i>0$, $(y)_i^\infty\neq \overline{\beta}$ and for all $j<0$, $(y)_{-\infty}^j\neq \overline{\beta}$. Then $y\in \Phi(X)$.
\end{lemma}
\begin{proof}
Let $(\Tilde{x}^n) $ be a sequence in $ X$ such that $\Phi(\Tilde{x}^n)\to y$. $X$ is a compact set so $(\Tilde{x}^n)$  has a convergent subsequence, $(\Tilde{x}^{n_k})$. Let the limit of $(x^{n_k})$ be $x\in X$.\\ 
    
    \noindent Define a new sequence $x^k=\Tilde{x}^{n_k}$. So we have, $x^k\to x$ and $\Phi(x^k)\to y$.  Now the goal is to show that $\Phi(x)=y$. Note that it is enough to show $\Phi(x)_0=y_0$ because for any $p\in \mathbb Z$ we know that $T^p(x^k)\to T^p(x)$ and $\sigma^p(x^k)\to \sigma^p(y)$. Applying the same proof as we did for $p=0$, we can show that $\Phi(x)_p=y_p$.\\
    
    \noindent The proof is divided into the following two cases: 
    \begin{itemize}
        \item If $y_0=\zero$ then we will show that $\Phi(x)_0=\zero$.   Since $\Phi(x^k)\to y$, there exists $N\in \mathbb N$ such that $\Phi(x^k)_0=\zero$ for all $k\geq N$ i.e. $x^k\in A^c$ for all $k\geq N$.  $A$ is open so $A^c$ is closed hence $x^k\to x\in A^c$. This shows that $\Phi(x)_0=\zero=y_0.$
        \item If $y_0=\one$ then we show that $\Phi(x)_0=\one$. We know that $(y)_{-\infty}^j\neq \overline{\one}$ for all $j<0$. That means there exists $l>0$ such that $y_{-l}=\alpha$. Since $\Phi(x^k)\to y$ that means $\sigma^{-l}(\Phi(x^k))=\Phi(T^{-l}(x^k))\to \sigma^{-l}(y)$. So we have $\sigma^{-l}(y)_0=\zero$ and $\sigma^{-l}(y)_{l}=\one$. So there exists $N\in \mathbb N$ such that $\Phi(T^{-l}(x^k))_{l}=\one$ for all $k\geq N$. So $T^{l}(T^{-l}(x^k))\in A$ for all $k\geq N$. There exists $i_k\geq 5$ and $j_k<i_k$ such that $T^{l}(T^{-l}(x^k))\in T^{j_k}\cylx{0^{i_k}}$ which means $T^{l-j_k}(T^{-l}(x^k))\in \cylx{0^{i_k}}$, i.e. $\Phi(T^{-l}(x^k))_{l-j_k+m}=\one$ for all $m\in \{0,..,i_k-1\}$. If $l-j_k<0<l-j_k+i_k-1$ then we would get $\Phi(T^{-l}(x))_0=\one$. However $y_{-l}=\zero$ so by the previous part we have $\Phi(x)_{-l}=\Phi(T^{-l}(x))_0=\zero$.  Thus $l-j_k>0$.\\
    \noindent Now we have, $j_k<l$ and $j_k<i_k$. Note that, $\cylx{0^{i_k}}\subseteq \cylx{0^{\min\{i_k,l\}}}$. Therefore, $x^k\in T^{j_k}\cylx{0^{i_k}}\subseteq T^{j_k}\cylx{0^{\min\{i_k,l\}}} $.\\
    \noindent Let $A_{l}=\bigcup_{i=5}^{l}\bigcup_{j=0}^{i-1}T^j\cylx{0^i}$. We have, $x^k\in A_{l}$ for all $k\geq N$, which is a closed set. Thus, $x^k\to x\in A_{l}\subseteq A$ which gives us $\Phi(x)_0=\one$.
    \end{itemize}
    
\end{proof}

\begin{theorem}\label{Y}
   Let $y\in \overline{\Phi(X)}$ then $y$ must be of at least one of the following types:-
\begin{description}

    \item[1)] $y\in \Phi(X)$.  
    \item[2)] there exists $i>0$, such that $(y)_i^\infty =\overline{\beta}$
    \item[3)] there exists $j<0$, such that $(y)_{-\infty}^j=\overline{\beta}$.
    
\end{description}
\end{theorem}
\begin{proof}
    The proof follows from Lemma \ref{4.2}.  
\end{proof}
\section{Ergodic measures on $(Y,\sigma)$}\label{mu}
Recall from section \ref{prelim} that $T$ is uniquely ergodic with unique invariant measure $\nu$. Let the push forward measure be $\mu=\Phi_*(\nu)$ on $Z$.  That is, $\mu(S)=\nu(\Phi^{-1}(S))$ for all $\nu$-measurable $\Phi^{-1}(S)$. In fact, $\mu(Y)=1$. The push forward of an ergodic measure is ergodic so $\mu$ is ergodic. Also note that for a sequence $x^n\in X$ such that $\mathcal{DN}_{2n}(T^{-n}(x^n))=0$ and $\mathcal{DN}_{2n+1}(T^{-n}(x^n))=1$ we have $\Phi(x^n)\to o$ where $o$ is the point of all $\one'$s. So the point $o$ belongs to $Y$.\\

\noindent Now our goal is to prove Theorem \ref{uniq} to show that $\mu$ and $\delta_o$ are the only ergodic measures on $(Y,\sigma)$.\\

\begin{lemma}\label{generic}
    Let $y$ be a forward and a backward generic point of an ergodic invariant measure $\mu'\neq \delta_o$ then $y\in \Phi(X)$.
\end{lemma}
\begin{proof}
    Note that $y$ is not forward or backward generic for the $\delta_o$ measure so we have  $(y)_i^\infty\neq \overline{\one}$, $(y)_{-\infty}^j\neq \overline{\one}$ for all $i,j$. By Lemma \ref{4.2} we have $y\in \Phi(X)$.
\end{proof} 
\noindent In Theorem \ref{uniq} we approximate the set $A$ by the sets $A_k=\bigcup_{i=5}^{k}\bigcup_{j=0}^{i-1}T^j\cylx{0^i}$ so that $\textbf{1}_{A_k}$ is continuous. We show that for a generic point $y=\Phi(x)$ of an ergodic measure $\mu'\neq \delta_o$, the frequency of visits of $x$ to $A$ can be approximated by the frequency of visits to $A_k$. For this approximation we define a sequence of sets called the  \emph{error sets}
\begin{align*}
    E_k=\bigcup_{i=k+1}^{\infty}\bigcup_{j=0}^{i-1} T^j\cylx{0^i}
\end{align*} for $k\geq 5$.\\
\noindent Lemma \ref{5.1} gives an upper bound to the frequency of visits of $x$ to $E_k$. In  fact the upper bound decays to 0 as $k\to \infty$. This upper bound helps us with the estimation in Theorem \ref{uniq} as mentioned above.

\begin{lemma}\label{5.1}
    Let $y=\Phi(x)$ be a forward and backward generic point of an ergodic measure $\mu'\neq \delta_o$. For $k\geq 5$ we have,
    \begin{align*}
        \liminf_{N\to \infty}\frac{1}{N}\sum_{i=0}^{N-1} \textbf{1}_{E_k}(T^i (x))\leq h_k \hspace{2mm}\text{where } h_k= \frac{k+2}{2^k}.
    \end{align*}
\end{lemma}
\begin{proof}

In this proof first we will construct a sequence of times $n_p$ such that the sequence $(T^{n_p}(x))$ in the orbit of $x$ ends with successively more zeros. The number of zeros that $T^{n_p}(x)$ ends with will be denoted by $s_p$.\\ 

\noindent By Lemma \ref{4.1} we know that there exists $K$ such that $\sigma^K(y)_{-1}=\zero$ and $\sigma^{K}(y)_0=\one$. This means for some finite $K$ we know that $T^K(x)\in \cylx{10^{s_0}}$ for some $s_0\geq 5$. Since the limit is unchanged, we can replace $x$ by $T^K(x)$ and assume that $T^{-1}(x)\notin A$ and $x\in A$.\\

\noindent Recall the function $\kappa(x)$ from section \ref{prelim} that counts the number of zeros that $x$ ends with. Let $n_0=0$ and $s_0=\kappa(x)$. Now we recursively define the following sequences:
\begin{align*}
    n_{p+1}=n_p+2^{s_p} &\\
    s_{p+1}=\kappa(T^{n_{p+1}}(x)).
\end{align*}

\noindent Note that $s_p\neq \infty$ as $(y)_{i}^\infty\neq \overline{\one}$. Also, $T^{n_{p+1}}(x)=T^{n_p+2^{s_p}}(x)\in \cylx{0^{s_p+1}}$ so $s_{p+1}\geq s_{p}+1$.\\ 

\noindent \textbf{Claim 1:} $T^i(x)\notin \cylx{0^{s_{p}+1}}$ for all $n_p\leq i<n_{p+1}$.\\

\noindent Note that $T^{n_p}(x)\in \cylx{10^{s_p}}$ and  $T^{n_{p+1}}(x)=T^{n_p+2^{s_p}}(x)\in \cylx{0^{s_p+1}}$. From this we have, $\{\mathcal{DN}_{s_p+1}(T^i(x)):n_p\leq i< n_{p+1}\}=\{2^{s_p},\ldots,2^{s_p+1}-1\}$. So $\mathcal{DN}_{s_p+1}(T^i(x))\neq 0$ for all $n_p\leq i<n_{p+1}$, which proves that $T^i(x)\notin \cylx{0^{s_p+1}}$ for all $n_p\leq i<n_{p+1}$.\\

\noindent \textbf{Claim 2:} For all $0\leq i<n_p$, $\kappa(T^i(x))<n_p-i$ \\

\noindent Let $l=\kappa(T^i(x))$. Since $i<n_p$ by claim 1 we have $l\leq\ s_{p-1}<s_p$. Now note that $T^i(x)\in \cylx{0^l}$ that means $T^{n_p}(x)\in T^{n_p-i}\cylx{0^l}$, so we have \begin{align*}
    \mathcal{DN}_l(T^{n_p}(x))=n_p-i\mod 2^l.
\end{align*}Also $T^{n_p}(x)\in \cylx{0^{s_p}}$ and we already know that $l<s_p$ which means $\mathcal{DN}_l(T^{n_p}(x))=0$. So we have $n_p-i\geq 2^l>l$.\\

 \noindent We split the orbit of $x$ into the sets $O_p=\{T^i(x): n_p\leq i<n_{p+1}\}$.\\

\noindent \textbf{Claim 3:} $O_p\cap E_{l}=\phi$ for all $l\geq s_p$.\\

\noindent Note that the error sets are nested so it is enough to show that $O_p\cap E_{s_p}=\phi$. Let $q\geq 0$ be such that $T^q(x)\in E_{s_p}$. That means $T^{q}(x)\in T^l\cylx{0^m}$ for some $m\geq s_p+1$ and $0\leq l<m$. So we have $T^{q-l}(x)\in \cylx{0^m}$. Now note that $q-l\geq 0$ because otherwise we have $T^{-1}(x)\in \{T^{q-l}(x),\ldots T^q(x)\}\subseteq A$ which is not true as $T^{-1}(x)\notin A$. Now using claim 1 and claim 2 we can say that $q-l\geq n_{p+1}$, so $q\geq n_{p+1}$. Therefore $O_p\cap E_{s_p}=\phi$.  \\

\noindent Finally we will prove the Lemma using this claim. The proof strategy is to calculate an upper bound for each $\left|O_p\cap E_k\right|$ and add them to calculate the number of intersections with the entire orbit.\\

             \noindent As $(s_p)$ is an increasing sequence there exists $p_0$ such that $s_{p_0-1}< k$ and $s_{p_0}\geq  k$. For $p\leq p_0$ by claim 3 we have $|O_p\cap E_k|=0$ so we will calculate an upper bound for $|O_p\cap E_k|$ for $p>p_0$. Again from claim 3 we have  $O_p\cap E_k=O_p\cap (E_k\backslash E_{s_p})$. Note that, $E_k\backslash E_{s_p}\subseteq \bigcup_{i=k+1}^{s_p}\bigcup_{j=0}^{i-1}T^j(\cylx{0^i})$. \\

             \noindent To calculate an upper bound for $\left|O_p\cap (E_k\backslash E_{s_p})\right|$ note that $T^{n_p}(x)\in \cylx{1 0^{s_p}}$ so for all $i<s_p$ we have $T^{n_p}(x)\in \cylx{0^i}$. The indices between $n_p$ and $n_p+2^{s_p}-1$ are of the form $n_p+c$ for some $0\leq c<2^{s_p}-1$.
             \noindent Now we can see that, $T^{n_p+c}(x)\in \cylx{0^i}$ if and only if $c$ is a multiple of $2^i$. The number of multiples of $2^i$ in $\{0,\ldots 2^{sp}-1\}$ is $2^{s_p-i}$. Also, for $c$ a fixed multiple of $2^i$ we have $T^{n_p+c+j}\in T^j\cylx{0^i}$ for all $0\leq j<i$.\\
             
             \noindent So we have
             \begin{align*}
                 &\left|O_p\cap \bigcup_{j=0}^{i-1} T^j\cylx{0^i}\right|\leq i\times 2^{s_p-i}=\frac{2^{s_p}.i}{2^i}\\ 
                 &\implies |O_p\cap E_k|\leq \sum_{i=k+1}^{s_p}\frac{2^{s_p}.i}{2^{i}}\\
                 &\leq 2^{s_p}.\frac{k+2}{2^k}=2^{s_p}h_k.
             \end{align*}
             
             \noindent  So we have $|O_p\cap E_k|\leq 2^{s_p}h_k$ for all $p\geq 0$.\\
             
             \noindent  Taking the average over the orbit we have,
             \begin{align*}
                 &\frac{1}{n_p}\sum_{i=0}^{n_p-1}\textbf{1}_{E_k}(T^i(x))\\
                 &=\frac{1}{n_p}\left(\sum_{i=0}^{n_1-1}\textbf{1}_{E_k}(T^i(x))+\sum_{i=n_1}^{n_2-1}\textbf{1}_{E_k}(T^i(x))+...+\sum_{i=n_{p-1}}^{n_{p}-1}\textbf{1}_{E_k}(T^i(x))\right)\\
                 & \leq  \frac{1}{n_p}\left(2^{s_0}h_k+..+2^{s_{p-1}}h_k\right)
                 \end{align*}
                 \noindent Note that $n_p=\sum_{m=0}^{p-1}2^{s_m}$, so we get,
             \begin{align*}
                 \frac{1}{n_p}\sum_{i=0}^{n_p-1}\textbf{1}_{E_k}(T^i(x))\leq h_k
             \end{align*}

             \noindent We know that limit inferior of a sequence is bounded above by the limit inferior of a subsequence, so we have
             
             \begin{align*}
                 \liminf_{N\to \infty}\frac{1}{N}\sum_{i=0}^{N-1}\textbf{1}_{E_k}(T^i(x))\leq h_k.
             \end{align*}
             \end{proof}
             \noindent Now we will show that there are only two ergodic measures on $(Y,\sigma)$ using Lemma \ref{5.1} to calculate error estimates.
             \begin{theorem}\label{uniq}
                 The system $(Y,\sigma)$ has exactly two ergodic measures $\mu$ and $\delta_o$. 
             \end{theorem}
             
             \begin{proof}
                Let $\mu'\neq\delta_o$ be an ergodic measure. We shall show that for a cylinder $C$ of $Y$ we must have $\mu'(C)=\mu(C)$. We use the indicator functions $\textbf{1}_C$ of the cylinders and the Birkhoff ergodic theorem for the proof.\\
                
                \noindent Let $y$ be a forward and backward generic point of $\mu'$. From Lemma \ref{generic} there exists $x\in X$ such that $y=\Phi(x)$. Since $\nu$ is uniquely ergodic on $(X,T)$ \cite{zbMATH01542660},  $x$ is a generic point for $\nu$.\\
                 \noindent For a cylinder set $C$, the indicator function of the cylinder set $\textbf{1}_C$ is continuous.\\
                 \noindent So since $y$ is a generic for $\mu'$ we have,  
                 
                 \begin{align}\label{mu1}
                     \lim_{N\to \infty}\frac{1}{N}\sum_{i=0}^{N-1} \textbf{1}_C(\sigma^i (y))=\mu'(C).
                 \end{align}
                 
                 \noindent Also,
                 \begin{align*}
             \sigma^n(y)\in C\iff \sigma^n(\Phi(x))\in C\iff \Phi(T^n(x))\in C\iff T^n(x)\in \Phi^{-1}(C).
         \end{align*}
         So we get, 
         \begin{align*}
             \frac{1}{N}\sum_{i=0}^{N-1} \textbf{1}_{C}(\sigma^i (y))=\frac{1}{N}\sum_{i=0}^{N-1} \textbf{1}_{\Phi^{-1}(C)}(T^i (x)).
         \end{align*}
         Since $ \lim_{N\to \infty}\frac{1}{N}\sum_{i=0}^{N-1} \textbf{1}_C(\sigma^i (y))$ exists this implies that the right hand limit also exists and we have the following,
         \begin{align}\label{0}
             \lim_{N\to \infty}\frac{1}{N}\sum_{i=0}^{N-1} \textbf{1}_{C}(\sigma^i (y))=\lim_{N\to \infty}\frac{1}{N}\sum_{i=0}^{N-1} \textbf{1}_{\Phi^{-1}(C)}(T^i (x)).
         \end{align}
        \noindent Now that we know that the limit exists, we will calculate this limit to show that $\mu'(C)=\mu(C)$ for an arbitrary cylinder $C$. This will show that $\mu'=\mu$.\\
         
         
         \noindent Let $C=\cylz{\omega_0\ldots\omega_{l-1}}$ be an arbitrary cylinder. The sets $A_k=\bigcup_{i=5}^{k}\bigcup_{j=0}^{i-1} T^j\cylx{0^i}$ are an increasing sequence of sets in $X$ that converge to $A$. So we have $\lim_{k\to \infty}\nu(A_k)\to \nu(A)$.\\
         
         \noindent Let $F_k=A\backslash A_k$ and recall the error sets $E_k$ from Lemma \ref{5.1}. Note that, $F_k\subseteq E_k=\bigcup_{i=k+1}^{\infty}\bigcup_{j=0}^{i-1} T^j\cylx{0^i}$ .\\

         \noindent  For notational convenience we will introduce the following notations. Let $A^\one=A$ and $A^\zero=A^c$. Let $A_k^\one=A_k$ and $A_k^\zero=A_k^c$. Note that, the sets $A_k^\one$ increase to $A^\one$ and the sets $A_k^\zero$ decrease to $A^\zero$. Let $B=\Phi^{-1}(C)$.\\
         
         \noindent From equation (\ref{0}) we know that $\lim_{N\to \infty}\frac{1}{N}\sum_{i=0}^{N-1} \textbf{1}_{B}(T^i (x))$ exists. Combining equations (\ref{0}) and (\ref{mu1}) we get
         \begin{align}\label{lim}
             \lim_{N\to \infty}\frac{1}{N}\sum_{i=0}^{N-1} \textbf{1}_{C}(\sigma^i (y))=\lim_{N\to \infty}\frac{1}{N}\sum_{i=0}^{N-1} \textbf{1}_{B}(T^i (x))=\mu'(C).
         \end{align}
         
         \noindent Note that $C=\cylz{\omega_0\ldots\omega_{l-1}}=\bigcap_{j=0}^{l-1}\sigma^{-j}(
         \cylz{\omega_{j}})$. So $\Phi^{-1}(C)=B=\bigcap_{j=0}^{l-1}T^{-j}(A^{\omega_j})$ where $\omega_j\in \{\zero,\one\}$.\\

         \noindent We know that $A^\one\Delta A^\one_k\subseteq E_k$. Let $B_k=\bigcap_{j=0}^{l-1}T^{-j}(A_k^{\omega_j})$.\\

         \noindent We will prove the theorem using the following claims.\\

         \noindent \textbf{Claim 1:} $\liminf_{N\to \infty} \left( \sum_{j=0}^{l-1}\left(\frac{1}{N}\sum_{i=0}^{N-1} \textbf{1}_{T^{-j}(E_k)}(T^{i}(x))\right)\right)\leq lh_k$ where $h_k$ is the bound in Lemma \ref{5.1}.\\

         \noindent Since \begin{align*}
             \left|\left( \sum_{j=0}^{l-1}\left(\frac{1}{N}\sum_{i=0}^{N-1} \textbf{1}_{T^{-j}(E_k)}(T^{i}(x))\right)\right)-l\cdot\frac{1}{N}\sum_{i=0}^{N-1}\textbf{1}_{E_k}(T^i(x))\right|\to 0 \text{ as $N\to \infty$.}
         \end{align*} \\
        \noindent We have 
         \begin{align*}
             &\liminf_{N\to \infty} \left( \sum_{j=0}^{l-1}\left(\frac{1}{N}\sum_{i=0}^{N-1} \textbf{1}_{T^{-j}(E_k)}(T^{i}(x))\right)\right)\\
             &= l\cdot\liminf_{N\to \infty} \frac{1}{N}\sum_{i=0}^{N-1}\textbf{1}_{E_k}(T^i(x))\\
             &\leq lh_k
         \end{align*}
         
         \noindent \textbf{Claim 2:} $B\Delta B_k\subseteq \bigcup_{j=0}^{l-1} T^{-j}(E_k)$.\\
         
         \noindent We prove the claim as follows.\\
         \noindent Let $a\in B\Delta B_k=(B\cap B_k^c)\cup (B^c\cap B_k)$. Suppose $a\in (B\cap B_k^c)$. Note that $(B\cap B_k^c)=\bigcap_{j=0}^{l-1}T^{-j}(A^{\omega_j})\cap (\bigcup_{j=0}^{l-1}T^{-j}(A_k^{\omega_j})^c)$. That means for some $0\leq j_0<l$, $a\in T^{-j_0}(A^{\omega_{j_0}})$. Also $a\in T^{-j_0}((A_k^{\omega_{j_0}})^c)$.\\
         \noindent Thus $a\in T^{-j_0}(A^{\omega_{j_0}})\cap T^{-j_0}((A_k^{\omega_{j_0}})^c)=T^{-j_0}(A^{\omega_{j_0}}\Delta A_k^{\omega_{j_0}})=T^{-j_0}(E_k)$, i.e. $a\in \bigcup_{i=0}^{l-1}T^{-j}(E_k)$. For $a\in B^c\cap B_k$ we can show by similar calculations that $a\in \bigcup_{i=0}^{l-1}T^{-j}(E_k)$. This proves that $B\Delta B_k\subseteq \bigcup_{j=0}^{l-1} T^{-j}(E_k)$.\\

         \noindent So we have,
         \begin{align}\label{ineq1}
             |\textbf{1}_{B}-\textbf{1}_{B_k}|=\textbf{1}_{B\Delta B_k}\leq \sum_{j=0}^{l-1} \textbf{1}_{T^{-j}(E_k)}.
         \end{align}
         \noindent \textbf{Claim 3:} \begin{align*}
            \liminf_{N\to \infty}\left|\frac{1}{N}\sum_{i=0}^{N-1}\textbf{1}_B(T^i(x))-\frac{1}{N}\sum_{i=0}^{N-1}\textbf{1}_{B_k}(T^i(x))\right|\leq lh_k.
         \end{align*}
      
         \noindent From the triangle inequality and claim 1 and claim 2 we have,
         \begin{align*}
             &\liminf_{N\to \infty}\left|\frac{1}{N}\sum_{i=0}^{N-1}\textbf{1}_B(T^i(x))-\frac{1}{N}\sum_{i=0}^{N-1}\textbf{1}_{B_k}(T^i(x))\right|\\
             &\leq \liminf_{N\to \infty}\frac{1}{N}\sum_{i=0}^{N-1} |\textbf{1}_B-\textbf{1}_{B_k}|(T^{i}(x))\\
             &\leq lh_k.
         \end{align*}
         \noindent Now we will prove the theorem using these claims.\\
         
         \noindent $A_k^\one, A_k^\zero$ are clopen sets so $\textbf{1}_{A_k^{\zero}}$ and $\textbf{1}_{A_k^{\one}}$ are both continuous functions which means means, $\textbf{1}_{B_k}=\Pi_{j=0}^{l-1}\textbf{1}_{T^{-j}(A^{i_j}_k)}$ is also continuous. As $x$ is a generic point we have
         \[\lim_{N\to \infty}\frac{1}{N}\sum_{i=0}^{N-1} \textbf{1}_{B_k}(T^i(x))=\nu(B_k).\]
          Using this in claim 3 we get
         \begin{align}\label{ineq3}
             \liminf_{N\to \infty}\left|\frac{1}{N}\sum_{i=0}^{N-1}\textbf{1}_B(T^i(x))-\nu(B_k)\right|\leq lh_k
         \end{align}

         \noindent From equation (\ref{lim}) we have
         \begin{align*}
             \lim_{N\to \infty}\frac{1}{N}\sum_{i=0}^{N-1} \textbf{1}_B(T^i (x))=\mu'(C).
         \end{align*}
         Using this in the inequality (\ref{ineq3}) we have
         \begin{align}\label{ineq4}
             \left|\lim_{N\to \infty}\frac{1}{N}\sum_{i=0}^{N-1}\textbf{1}_B(T^i(x))-\nu(B_k)\right|=|\mu'(C)-\nu(B_k)|\leq lh_k
         \end{align}

         \noindent Since $B\Delta B_k\subseteq \bigcup_{j=0}^{l-1}T^{-j}(A\Delta A_k)$ and $\nu(A\Delta A_k)\to 0$ as $k\to \infty$. Also $T$- is $\nu$ invariant, so we have $\nu(B\Delta B_k)\to 0$ i.e. $\lim_{k\to \infty}\nu(B_k)=\lim_{k\to \infty}\nu(B)$. Using this in inequality (\ref{ineq4}) we have 
         \begin{align*}
             \mu'(C)=\nu(B).
         \end{align*}
          \noindent Recall that $B=\Phi^{-1}(C)$ so we have $\nu(B)=\mu(C)$. So for all cylinders $C$ we have $\mu'(C)=\mu(C)$. Thus we conclude that there are only two ergodic measures $\mu$ and $\delta_o$ on $(Y,\sigma)$.
         \end{proof}

\section{Construction of the potential}\label{pot}

\noindent Recall that the distance between a point $z\in Z$ and a closed set $S\subseteq Z$ is given by
\begin{align*}
    \text{dist}(z,S)=\inf\{d(z,s):s\in S\}.
\end{align*}

\noindent Now we construct a potential $\psi$ on $Z$ that has equilibrium states supported on $Y$. The following is a rough sketch of such a potential that decreases with increase in distance from $Y$. Similar constructions have been done in \cite{antonioli2016compensation} and \cite{zbMATH07358505}.

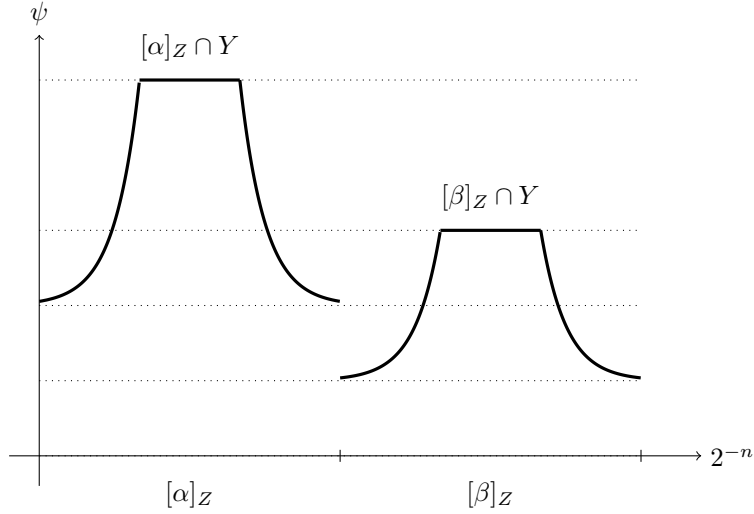
\begin{tikzpicture}[scale=4, domain=0:2, samples=100]
  \draw[->] (-0.1,0) -- (2.2,0) node[right] {$2^{-n}$};
  \draw[->] (0,-0.1) -- (0,1.4) node[above] {$\psi$};

  \draw (1,0.02) -- (1,-0.02);
  \draw (2,0.02) -- (2,-0.02);

  \draw[very thick] (1/3,1.25) -- (2/3,1.25);
  \draw[very thick, domain=0:1/3] plot (\x, {0.5 + 0.75 * exp((\x - 1/3)*12)});
  \draw[very thick, domain=2/3:1] plot (\x, {0.5 + 0.75 * exp((2/3 - \x)*12)});

  \draw[very thick] (4/3,0.75) -- (5/3,0.75);
  \draw[very thick, domain=1:4/3] plot (\x, {0.25 + 0.5 * exp((\x - 4/3)*12)});
  \draw[very thick, domain=5/3:2] plot (\x, {0.25 + 0.5 * exp((5/3 - \x)*12)});

  \node[left] at (0,1.25) {};
  \node[left] at (0,0.75) {};

  \draw[dotted] (0,1.25) -- (2,1.25); 
  \draw[dotted] (0,0.75) -- (2,0.75);
  \draw[dotted] (0,0.5) -- (2,0.5);
  \draw[dotted] (0,0.25) -- (2,0.25);
  \draw[dotted] (0,0) -- (2,0);       

  \node[below=6pt] at (0.5, 0) {$\cylz{\zero}$};
  \node[below=6pt] at (1.5, 0) {$\cylz{\one}$};

  \node[above=4pt] at (0.5, 1.25) {$\cylz{\zero} \cap Y$};
  \node[above=4pt] at (1.5, 0.75) {$\cylz{\one} \cap Y$};

\end{tikzpicture}

\noindent Define potentials $\psi_0: Z\to \mathbb R$, $\psi_1:Z\to \mathbb R$ and $\psi:Z\to \mathbb R$ as follows:-\\

         \begin{align*}
             \psi_0(z) = \begin{cases}
        -2 & \text{if } z\in \cylz{\one} \\
        0 & \text{if } z\in \cylz{\zero} \\
        \end{cases}
        \end{align*}
         and
         \begin{align*}
         \psi_1(z) = \begin{cases}
        -\frac{12}{\sqrt{n}} & \text{if dist}(z,Y)=2^{-n}  \\
        0 & \text{if } z\in Y \\
        \end{cases}
        \end{align*}
        Define $\psi=\psi_0+\psi_1$. This potential decreases  with increasing distance from $Y$. \\

        \noindent Note that $\psi_0$ is continuous on $Z$. Now suppose for $z\in Z$ we have $\text{dist}(z, Y)=2^{-n}$ then, for an $(n+1)$-cylinder $C$ containing $z$ for all $z'\in C$ we have $\text{dist}(z',Y)= 2^{-n} $. So we have
        \begin{align*}
            \left|\psi_1(z)-\psi_1(z')\right|=0
        \end{align*} This shows that $\psi_1$ is continuous, hence $\psi=\psi_0+\psi_1$ is continuous.\\

        \noindent We denote the pressure of $\psi$ by $P(\psi)$. By the variational principle \cite{zbMATH01542660} we have,
                \begin{align*}
                    P(\psi)=\sup_{\eta} \left(h_{\eta}(\sigma)+\int \psi\, d\eta\right)
                \end{align*} where $h_\eta(\sigma)$ is the entropy of $\sigma$ with respect to the measure $\eta$ and the supremum is taken over all the $\sigma$- invariant measures. \\

                \noindent  Note that $\mu$  is the push-forward of $\nu$ so we have $h_{\mu}(\sigma)\leq h_\nu(T)$. Recall from section \ref{prelim} that  $h_\nu(T)=0$.  So we have $h_\mu(\sigma)=0$.\\

                \noindent We have $\int\psi\, d\mu =-2\mu(\cylz{\one})$.  Therefore 
                \begin{align*}
                    h_{\mu}(\sigma)+\int \psi\, d\mu=-2\mu(\cylz{\one}).
                \end{align*} Now our goal is to show that $\mu$ is an equilibrium state of $\psi$ i.e. to show that $P(\psi)=-2\mu(\cylz{\one})$.\\

                \noindent First we will prove a few Lemmas which will help us prove Theorem \ref{exist} where we show that $\mu$ is an equilibrium state.\\

                \noindent Recall from section \ref{prelim} that $\mathcal{L}_n(Y)$ denotes the set of words of length $n$ in $Y$. Also $\mathcal{L}(Y)$ denotes the language of $Y$ i.e. the set of words of all possible lengths.

                \begin{lemma}\label{exist1}
                    Let $d\in \mathbb N$. The number of words in $\mathcal{L}_d(Y)$ is at most $2(d+1)^3$.
                \end{lemma}
                \begin{proof}
                    Let $d\in \mathbb N$ and $k$ such that $2^k\leq d<2^{k+1}$, we will give an upper bound for $|\mathcal{L}_{d}(Y)|$.\\
                \noindent Let $ W\in \mathcal{L}_d(Y)$ be an arbitrary word. Let $y\in \Phi(X)\cap W$ and let $y=\Phi(x)$. So that $(y)_0^{d-1}=W$. \\

                \noindent Let $p_0=\max\{0\leq i<d:x,\ldots,T^i(x)\in E_{k}\}$ otherwise $p_0=-1$ if $\{x,\ldots,T^i(x)\}\cap E_k=\phi$. Let $q_0=\min\{p_0< i< d: T^i(x)\in E_k\}$ otherwise $q_0=d$ if $\{x,\ldots,T^i(x)\}\cap E_k=\phi$. So we have $T^{q_0}(x)\in \cylx{0^{k+1}}$. Note that the time at which $x$ visits $E_k$ after $q_0$ is at $q_0+2^{k+1}>d$. This implies that $\{T^{p_0+1}(x),\ldots,T^{q_0-1}(x)\}\cap E_k=\phi$. So we can say that $T^{p_0+1}(x)\in T^j\cylx{0^{k+1}}$ where $k<j\leq 2^{k+1}-(q_0-p_0)\leq 2^{k+1}-d$. Since the segment $(y)_{p_0+1}^{q_0-1}$ is determined by $\{T^{p_0+1}(x),\ldots,T^{q_0-1}(x)\}$, the number of choices for $(y)_{p_0+1}^{q_0-1}$ is bounded above by $2^{k+1}-d-k\leq 2d$. Now note that the number of choices for $p_0$ and $q_0$ are both $d+1$ so the total number of words in $\mathcal{L}_d(Y)$ is bounded above by $2d(d+1)^2\leq 2(d+1)^3$.
                \end{proof}
                \begin{lemma}\label{calc}
                    The cylinder set $\cylz{\one}$ satisfies
                    \begin{align*}
                        \frac{5}{32}\leq \mu(\cylz{\one})\leq\frac{6}{32}.
                    \end{align*}
                    
                \end{lemma}
                \begin{proof}
                    Let $L_i=\bigcup_{j=0}^{i-1} T^j\cylx{0^i}$. We define sets $B_m$ as $ B_5=\bigcup_{j=0}^{4} T^{j}\cylx{0^5}$ and  $B_m=L_m\backslash \bigcup_{i=5}^{m-1} L_i$ for $m\geq 6$.\\
                    \noindent For $m\geq 6$ the set $B_m$ is either the cylinder set $T^{m-1}\cylx{0^m}$ or the set is empty.\\
                \noindent The set is empty if and only if $T^{m-1}\cylx{0^m}$ is contained in $\bigcup_{i=5}^{m-1} L_i$ which happens if $(m-1)\mod 2^i\in \{0,\ldots,i-1\}$ for some $5\leq i\leq m-1$. \\

                \noindent We have $\nu(B_5)=\tfrac{5}{32}$ and $\nu(B_m)=\frac{q_m}{2^m}$ where $q_m\in\{0,1\}$ for $m\geq 6$.\\

                \noindent Note that $A=\bigsqcup_{m=5}^\infty B_m$ this means $\nu(A)=\sum_{m=5}^\infty\nu(B_m)$. Since
                \begin{align*}
                         \frac{5}{32}\leq\sum_{m=5}^\infty\nu(B_m)\leq \frac{6}{32}.
                \end{align*} We have $\frac{5}{32}\leq\nu(A)=\mu(\cylz{\one})\leq \frac{6}{32}$.
                
                \end{proof}
                \begin{lemma}\label{exist2}
                    Let $d\in \mathbb N$ and let $y\in \Phi(X)\cap \mathcal{L}_d(Y)$. We have 
                    \begin{align*}
                        S_d\psi_0(y)\leq -2\mu(\cylz{\one})d+ 6+2\log_2d
                    \end{align*}
                \end{lemma}
                \begin{proof}
                \noindent Let $y=\Phi(x)$ and $d\geq 64$. By definition of $\psi_0$ we have 
                \begin{align*}
                    &S_d\psi_0(y)=-2\times\left|\{0\leq i<d:y_i=\one\}\right|\\
                    &=-2\times|\{0\leq i<d:T^i(x)\in A\}|.
                \end{align*}

                \noindent From Lemma \ref{calc} recall the sets   $L_i=\bigcup_{j=0}^{i-1} T^j\cylx{0^i}$ for $i\geq 5$ and $B_m$ defined as $ B_5=\bigcup_{j=0}^4 T^j\cylx{0^5}$ and  $B_m=L_m\backslash \bigcup_{i=5}^{m-1} L_i$ for $m\geq 6$.\\
                
                \noindent Also we have $\nu(B_5)=\tfrac{5}{32}$ and $\nu(B_m)=\frac{q_m}{2^m}$ where $q_m\in\{0,1\}$ for $m\geq 6$.\\
                
                \noindent Note that the visits of $T^i(x)$ to $B_m$ are periodic with period $2^m$. This can be seen from the properties of the odometer map given in section \ref{prelim}.  This means for $m\geq 6$ the number of $i$ in a span of $2^m$ steps such that $T^i(x)\in B_m$ is $q_m$. For $B_5$, in a span of $2^5$ steps the number of $i$ such that $T^i(x)\in B_5$ is $5$. Therefore in $d$ steps we have,
                \begin{align*}
                    |\{0\leq i< d:T^i(x)\in B_m\}|\geq 5\left\lfloor \frac{d}{32}\right\rfloor  \text{ \hspace{5mm}for  } m=5
                \end{align*}
                and
                \begin{align*}
                    |\{0\leq i< d:T^i(x)\in B_m\}|\geq \left\lfloor \frac{d}{2^m}\right\rfloor q_m \text{ \hspace{5mm}for  } m\geq 6
                \end{align*}
                \noindent Recall the sets $A_k=\bigcup_{i=5}^\infty\bigcup_{j=0}^{i-1}T^j\cylx{0^i}$. We have  $A_k=\bigsqcup_{m=5}^k B_m$ for $k\geq 5$ and  $A=\bigsqcup_{m=5}^\infty B_m$. Also, 
                \begin{align*}
                    |\{0\leq i< d:T^i(x)\in A\}|=\sum_{m=5}^\infty|\{i\leq d:T^i(x)\in B_m\}|.
                \end{align*}
                Therefore
                \begin{align*}
                     &|\{0\leq i< d:T^i(x)\in A\}|\\
                     &\geq 5\cdot\left\lfloor \frac{d}{32}\right\rfloor +\sum_{m=6}^{\lfloor\log_2d\rfloor} \left\lfloor \frac{d}{2^m}\right\rfloor q_m\\
                     &\geq 5\cdot\left(\frac{d}{32}-1\right)+\sum_{m=6}^{\lfloor\log_2d\rfloor} \left(\frac{d}{2^m}-1\right) q_m\\
                     &=\frac{5d}{32}+ \sum_{m=6}^{\lfloor\log_2d\rfloor} \left(\frac{dq_m}{2^m}\right)-5-\sum_{m=6}^{\lfloor\log_2d\rfloor}q_m. 
                \end{align*}
                Also $\nu(B_m)=\frac{5}{32}$ for $ m=5$ and $\nu(B_m)=\frac{q_m}{2^m}$ for $m\geq 6$ so we can rewrite the above inequality as,
                \begin{align*}
                    &|\{0\leq i< d:T^i(x)\in A\}|\\
                    &\geq d\left(\sum_{m=6}^{\lfloor\log_2d\rfloor} \frac{q_m}{2^m}+\frac{5}{32}\right)-\left(\sum_{m=6}^{\lfloor\log_2d\rfloor}q_m+5\right)\\
                    &= d\sum_{m=5}^{\lfloor\log_2d\rfloor} \nu(B_m)-\left(\sum_{m=6}^{\lfloor\log_2d\rfloor}q_m+5\right)\\
                    & \geq d\left(\nu\left(\bigsqcup_{m=5}^{\lfloor\log_2d\rfloor} B_m\right)\right)-\log_2d\\
                    &=\nu( A_{\lfloor\log_2d\rfloor})d-\log_2 d
                \end{align*}
               
                \noindent We have $\nu(A)=\mu(\cylz{\one})$ and 
                \begin{align*}
                    \nu(A\backslash A_{\lfloor\log_2d\rfloor})\leq\sum_{m=\lfloor\log_2d\rfloor}^\infty\frac{q_m}{2^m}\leq\frac{2}{d}
                \end{align*}
                \noindent Hence
                \begin{align}
                    |\{0\leq i<d:T^i(x)\in A\}|\geq \mu(\cylz{\one})d-2-\log_2 d.
                \end{align}
                Therefore, \begin{align*}
                    &-2\times |\{0\leq i<d:T^i(x)\in A\}|\\
                    &\leq -2{\mu(\cylz{\one})}d+4+2\log_2d
                    \end{align*}
                    So we have
                    \begin{align*}
                    S_d\psi_0(y)\leq -2{\mu(\cylz{\one})}d+4+2\log_2d
                    \end{align*}

                    \noindent Now note that for $32\leq d<64$ we have \begin{align*}
                    |\{0\leq i< d:T^i(x)\in B_m\}|\geq 5\left\lfloor \frac{d}{32}\right\rfloor =5 \text{ \hspace{5mm}for  } m=5
                \end{align*}
                and
                \begin{align*}
                    |\{0\leq i< d:T^i(x)\in B_m\}|\geq \left\lfloor \frac{d}{2^m}\right\rfloor q_m =0\text{ \hspace{5mm}for  } m\geq 6.
                \end{align*} 
                This means \begin{align*}
                    |\{0\leq i< d:T^i(x)\in A\}|=\sum_{m=1}^\infty|\{i\leq d:T^i(x)\in B_m\}|\geq 5.
                \end{align*}
                So we have
                \begin{align*}
                    & S_d\psi_0(y)=-2\times |\{0\leq i<d:T^i(x)\in A\}|\leq -10
                \end{align*}
                By Lemma \ref{calc} we know that $\mu(\cylz{\one})\leq \frac{6}{32}$. From this we have
                \begin{align*}
                    S_d\psi_0(y)\leq -2{\mu(\cylz{\one})}d+4+2\log_2d
                    \end{align*}
                    \noindent Finally for $d<32$ we have \begin{align*}
                    |\{0\leq i< d:T^i(x)\in B_m\}|\geq 0 \text{ \hspace{5mm}for  } m\geq 1.
                \end{align*} 
                So the upper bound
                \begin{align*}
                    S_d\psi_0(y)\leq -2{\mu(\cylz{\one})}d+4+2\log_2d
                    \end{align*}
                    is true for all $d\in \mathbb N$.

                \end{proof}
                
                \begin{lemma}\label{exists3}
                Let \begin{align*}
                    Q_n^l=\sum_{n_1+..+n_l=n}\,\prod_{i=1}^{l}\left(\sum_{W_i\in \mathcal{L}_{n_i}(Y)} \left(e^{ S_{n_i}\psi_0(y_{W_i})-12\sqrt{n_i}}\right)\right)
                \end{align*}
                
                where $y_{W_i}\in \Phi(X)\cap \cylz{W_i}$ is an arbitrary element.
                     We have
                    \begin{align*}
                    Q_n^l\leq \left(\frac{1}{2}\right)^le^{-2\mu(\cylz{\one})n}
                \end{align*}
                    
                \end{lemma}
                \begin{proof}

                     \noindent From Lemma \ref{exist1} we know that $\left|\mathcal{L}_{n_i}(Y)\right|\leq 2(n_i+1)^3$.
                     
                     \noindent From Lemma \ref{exist2} we have 
                     \begin{align*}
                    S_d\psi_0(y_{W_i})\leq -2{\mu(\cylz{\one})}n_i+4+2\log_2n_i
                    \end{align*}
                    So we have the following 
                    \begin{align*}
                    &Q_n^l=\sum_{n_1+..+n_l=n}\,\prod_{i=1}^{l}\left(\sum_{W_i\in \mathcal{L}_{n_i}(Y)} \left(e^{ S_{n_i}\psi_0(y_{W_i})-12\sqrt{n_i}}\right)\right)\\
                    &\leq \sum_{n_1+..+n_l=n}\,\prod_{i=1}^{l}\left(\sum_{W_i\in \mathcal{L}_{n_i}(Y)} \left(e^{ -2{\mu(\cylz{\one})}n_i+4+2\log_2n_i-12\sqrt{n_i}}\right)\right) \\
                    &=\sum_{n_1+\ldots n_l=n}\prod_{i=1}^l\left(\left|\mathcal{L}_{n_i}(Y)\right|e^{-2\mu(\cylz{\one})n_i-7\sqrt{n_i}}\right)\\
                    & \leq  e^{-2\mu(\cylz{\one})n}\sum_{n_1+..+n_l=n}\,\prod_{i=1}^{l}\left(e^{ -{2\sqrt{n_i}}}\right)\\
                        & \leq e^{-2\mu(\cylz{\one})n}.\,\prod_{i=1}^{l}\left(\sum_{m=1}^\infty e^{ -{2\sqrt{m}}}\right)
                \end{align*}
                    Now we will calculate a simple integral to give a suitable upper bound for the sum $\sum_{m=1}^\infty e^{ -{2\sqrt{m}}}$. Note that the function $f(x)=e^{-2\sqrt{x}}$ is a decreasing function therefore the right Riemann sum of the curve using the rectangles $\{(m,m+1]: m\in \mathbb N\}$ is less than the integral of the function. However, the right Riemann sum using these rectangles is, $\left(\sum_{m=2}^\infty e^{-2\sqrt{m}}\right)$.\\
                \noindent Calculating the integral we have 
                \begin{align*}
                    &\left(\sum_{m=1}^\infty e^{-2\sqrt{m}}\right)\\
                    &\leq e^{-{2}}+\int_1^{\infty} e^{-2\sqrt{x}}\,dx\\
                    &=\frac{1}{e^2}+\frac{3}{2e^2}< \frac{1}{2}.
                \end{align*}
                Thus we get 
                \begin{align*}
                    Q_n^l\leq \left(\frac{1}{2}\right)^le^{-2\mu(\cylz{\one})n}
                \end{align*}
                    
                \end{proof}
                    
        \begin{theorem}\label{exist}
            $\mu$ is an equilibrium state of $\psi$.
        \end{theorem}
        \begin{proof}
                
                \noindent For the full shift $(Z,\sigma)$ the pressure $P(\psi)$ is given by \cite{zbMATH01542660}(pg. 213 Theorem 9.6), 
                \begin{align*}
                    P(\psi)=\lim_{n\to \infty}\frac{1}{n}\log\sum_{W\in \mathcal{L}_n(Z)} e^{S_n\psi(z_W)}
                \end{align*} where for each $W\in \mathcal{L}(Y)$, $z_W$ is an  arbitrary element in $\cylz{W}$ and $S_n\psi(z_W)=\sum_{i=0}^{n-1}\psi(\sigma^i(z_W))$. \\

                \noindent Let 
                \begin{align*}
                    Q_n=\sum_{W\in \mathcal{L}_n(Z)} e^{S_n\psi(z_W)}
                \end{align*}
                
                \noindent To calculate $P(\psi)$ we will break $W\in \mathcal{L}_n(Z)$ into maximal words in $\mathcal{L}(Y)$ as follows.\\
                
                \noindent Let $z_W\in W$ be such that $(z_W)_n^{n+2}=\zero\one\zero$. We choose such a $z_W$ because this guarantees that $(z)_0^{n+2}$ does not belong to $\mathcal{L}(Y)$ as $\zero\one\zero\notin \mathcal{L}(Y)$ . The reason for this is that if for some $y\in Y$ we have $y_n=\zero$ and $y_{n+1}=\one$. Since $y$ is in the closure of $\Phi(X)$ so we have an $x\in X$ such that $\Phi(x)_n=\zero$ and $\Phi(x)_{n+1}=\one$. So we have $T^{n+1}(x)\in \cylx{0^m}$ where $m\geq 5$ which means $\Phi(x)_{n+1},\ldots\Phi(x)_{n+m}=\one$.This shows that if any word in $\mathcal{L}(Y)$ has at least one $\one$ following an $\zero$ then it must have at least five $\one'$s. So $\zero\one\zero$ is not a valid word in $\mathcal{L}(Y)$. \\
                
                \noindent Define $n'_1=\max\{j: (z_W)_0^{j-1}\in \mathcal{L}(Y)\}$ and $q'_1=n_1'$. If $n'_1< n$, we define $n'_2=\max\{j: (z_W)_{q'_1}^{q'_1+j-1}\in \mathcal{L}(Y)\}$ and $q'_2=q'_1+n'_2$ and so on, i.e. $n'_i=\max\{j: (z_W)_{q'_{i-1}}^{q'_{i-1}+j-1}\in \mathcal{L}(Y)\}$ and $q'_i=\sum_{j=1}^{i-1}n'_j$. We terminate at the step $l$ when $\sum_{i=1}^{l-1}n'_i< n$ and $\sum_{i=1}^ l n'_i\geq n$.\\
                
                \noindent Now let $n_i=n'_i$ for all $i<l$ and $n_l=n-\sum_{i=1}^{l-1} n_i$. Also let $q_i=\sum_{j=1}^{i-1}n_j$\\
                
                \noindent Let $W_i=(z_W)_{q_{i-1}}^{q_i-1}$ for all $i\leq l$. Note that $W_i\in \mathcal{L}_{n_i}(Y)$. So we have broken the word $W\in \mathcal{L}_n(Z)$ into the words $W_i\in \mathcal{L}_{n_i}(Y)$.\\

                \noindent {Given below is a pictorial representation of the strategy of breaking the word into maximal words:}
                
\[
 z_W=\ldots z_{-2}z_{-1}.\underbrace{z_0\ldots z_{n_1}z_{n_1+1}\ldots z_{n_2}z_{n_2+1}\ldots z_{n_{l-1}}\ldots z_l}_{\text{the word W }}\zero\one\zero\ldots
\]
\[
 z_W=\ldots z_{-2}z_{-1}.\overbrace{z_0\ldots z_{q_1-1}}^{W_1\in \mathcal{L}_{n_1}(Y)}\,\,\overbrace{z_{q_1}\ldots z_{q_2-1}}^{W_2\in \mathcal{L}_{n_2}(Y)}z_{q_2}\ldots\,\, \overbrace{z_{q_{l-1}}\ldots z_{q_{l}-1}}^{W_l\in \mathcal{L}_{n_l}(Y)}\,\,\zero\one\zero\ldots
\]
\[
 z_W=\ldots z_{-2}z_{-1}.\underbrace{z_0\ldots z_{q_1-1}z_{q_1}}_{\notin \mathcal{L}(Y)}\,\,{\ldots z_{q_2-1}}\,\,z_{q_2}\ldots\ \underbrace{z_{q_{l-1}}\ldots z_{q_l-1}\zero\one\zero}_{\notin \mathcal{L}(Y)}\ldots
\]
                
                \noindent  Note that
                \begin{align*}
                    S_n\psi(z_W)=S_{n_1}\psi(z_W)+S_{n_2}\psi(\sigma^{q_1}(z_W))+...+S_{n_l}\psi(\sigma^{q_{l-1}}(z_W)).
                \end{align*}
                \noindent Since $z_W\in W_1\in \mathcal{L}_{n_1}(Y)$ that means dist$(z_W, Y)\geq 2^{-n_1}$. Similarly dist$(\sigma^i(z_W), Y)\geq 2^{-n_1}$ for all $j< n_1$. So we have
                \begin{align*}
                    \psi_1(\sigma^j(z_W))\leq -\frac{12}{\sqrt{n_1}}
                \end{align*} for all $j<n_1$.\\
                
                \noindent So we have
                \begin{align*}
                    &S_{n_1}\psi(z_W)\leq S_{n_1}\psi_0(z_W) -12\sqrt{n_1}
                \end{align*}

                \noindent If $y_{W_1}\in \Phi(X)\cap \cylz{W_1}$, then $S_{n_1}\psi_0(z_W)=S_{n_1}\psi_0(y_{W_1})$ as $\psi_0$ depends only on the $0$-th position.\\
                
                \noindent From this we have
                \begin{align*}
                    S_{n_1}\psi(z_W)\leq S_{n_1}\psi_0(y_{W_1}) -12\sqrt{n_1}.
                \end{align*}

                \noindent We  are using $y_{W_1}$ instead of $z_W$ as the number of $\one'$s in $y_{W_1}$ is equal to the number of times the orbit of the preimage $x_{W_1}$ belongs to $A$, a fact that we will use later for the calculations.\\

                \noindent Similarly we have
                
                \begin{align}\label{est1}
                    S_{n_{i}}\psi(\sigma^{q_{i-1}}(z_W))\leq S_{n_{i}}\psi_0(y_{W_i})- 12\sqrt{ n_i}
                \end{align}
                for all $2\leq i<l$ and for $y_{W_i}\in \cylz{W_i}\cap \Phi(X)$.\\
                
                \noindent The calculation for the last term is different from that of $1\leq i<l$. We have
                \begin{align*}
                    S_{n_{l}}\psi(\sigma^{q_{l-1}}(z_W))\leq S_{n_{l}}\psi_0(\sigma^{q_{l-1}}(z_W))-\frac{12n_l}{\sqrt{n'_l}}
                \end{align*}
                \noindent We chose $z_W$ such that $(z_W)_{n}^{n+2}\notin \mathcal{L}(Y)$. This guarantees that $\sum_{i=1}^l n'_i<n+3$. That means $n_l'-n_l<3$, so we get
                  \begin{align*}
                      &S_{n_{l}}\psi(\sigma^{q_{l-1}}(z_W))\leq S_{n_{l}}\psi_0(\sigma^{q_{l-1}}(z_W)) -12\sqrt{n_l'}+\frac{36}{\sqrt{n_1'}}\\
                      & \leq S_{n_{l}}\psi_0(\sigma^{q_{l-1}}(z_W)) -12\sqrt{n_l}+36
                  \end{align*}

                  \noindent Again choose $y_{W_l}\in \cylz{W_l}\cap \Phi(X)$ so we have,
                  \begin{align}\label{est2}
                      S_{n_{l}}\psi(\sigma^{q_{l-1}}(z_W))<S_{n_{l}}\psi_0(y_{W_l})-12\sqrt{n_i}+36
                  \end{align}
                  \noindent Summing (\ref{est1}) from $1$ to $l-1$ and adding that to (\ref{est2}) we have,
                  \begin{align*}
                      S_n\psi(z_W)\leq \left(\sum_{i=1}^{l}\left(S_{n_i}\psi(y_{W_i})-12\sqrt{n_i}\right)\right)+36
                  \end{align*}
                  
                \noindent Let $K=e^{36}$, so we have,
                \begin{align*}
                    &Q_n\leq K\sum_{l}\left(\sum_{n_1+\ldots n_l=n} \left(\sum_{W_i\in \mathcal{L}_{n_i}(Y)}e^{\left(\sum_{i=1}^l S_{n_i}\psi_0(y_{W_i})-12\sqrt{n_i}\right)}\right)\right)
                    \end{align*}
                    Rearranging we get
                    \begin{align}\label{pressureineq}
                    Q_n\leq K\sum_l\left(\sum_{n_1+..+n_l=n}\,\prod_{i=1}^{l}\left(\sum_{W_i\in \mathcal{L}_{n_i}(Y)} \left(e^{ S_{n_i}\psi_0(y_{W_i})-12\sqrt{n_i}}\right)\right)\right).
                \end{align}

                \noindent Note that $Q_n^l=\sum_{n_1+..+n_l=n}\,\prod_{i=1}^{l}\left(\sum_{W_i\in \mathcal{L}_{n_i}(Y)} \left(e^{ S_{n_i}\psi_0(y_{W_i})-12\sqrt{n_i}}\right)\right)$.\\

                \noindent So we have 
                \begin{align}\label{rel}
                    Q_n\leq K\sum_l Q_n^l.
                \end{align} Therefore to estimate $Q_n$ we have split it according to the number of pieces $l$ in the maximal $\mathcal{L}(Y)$ decomposition. 
                \noindent From Lemma \ref{exists3} we know that 
                \begin{align*}
                    Q_n^l\leq \left(\frac{1}{2}\right)^le^{-2\mu(\cylz{\one})n}
                \end{align*}
                Using this in (\ref{pressureineq}) we have
                \begin{align*}
                    Q_n\leq K e^{-2\mu(\cylz{\one})n}\sum_{l=1}^\infty \left(\frac{1}{2}\right)^l=2K e^{-2\mu(\cylz{\one})n}
                \end{align*}
                Since \begin{align*}
                    P(\psi)=\lim_{n\to \infty}\frac{1}{n}\log Q_n.
                \end{align*}
                So we have
                \begin{align*}
                    P(\psi)\leq -2\mu(\cylz{\one})
                \end{align*}

                \noindent This means $P(\psi)\leq h_\mu(\sigma)+\int\psi\,d\mu$ where $\mu=\Phi_*\nu$ . So  by the variational principle, $P(\psi)=-2\mu(\cylz{\one})$ hence $\mu$ is an equilibrium state.

                \end{proof}
                \noindent Now we will prove two lemmas using the weak Birkhoff Theorem and the weak Shannon-McMillan-Breiman Theorem given in section \ref{prelim}. Finally in Theorem \ref{one} we will prove the uniqueness of the equilibrium state.

                \begin{lemma}\label{lemma1}
                    Let $\eta$ be ergodic on $Z$ and $\psi$ be a continuous function on $Z$. For $\epsilon>0$ there exists $N\in \mathbb N$ such that for all $n\geq N$, for any collection $C_1,\ldots,C_m$ of $n$-cylinders such that $\sum_{j=1}^m\eta(C_j)\geq \frac{1}{2}$ we have,
                    \[\frac{1}{n}\log\left(\sum_{j=1}^m e^{S_n\psi(z_j)}\right)>h_\eta+\int\psi\,d\eta-\epsilon\] where $z_j\text{ is an arbitrary element of } C_j$.
                \end{lemma}
                \begin{proof}
                    \noindent For $\epsilon>0$ from Theorem \ref{wb} we know that there exists $N_1$ such that for all $n\geq N_1$ we have
                    \begin{align*}
                        \eta\left(\left\{z:\left|\tfrac{S_n\psi(z)}{n}-\int\psi\,d\eta\right|<\tfrac{\epsilon}{6}\right\}\right)>\tfrac{7}{8}
                    \end{align*}
                    Let $G_1^n=\left\{z:\left|\frac{S_n\psi(z)}{n}-\int\psi\,d\eta\right|<\frac{\epsilon}{3}\right\}$. \\
                    
                    \noindent We define $\text{var}_k(\psi)=\max_{\text{dist}(z^1,z^2)\leq 2^{-k}}\left|\psi(z^1)-\psi(z^2)\right|$. Note that $\text{var}_k\psi$ is a decreasing sequence. Since $\psi$ is continuous on a compact set $\text{var}_k\psi\to 0$ as $k\to \infty$ . Also for $z^1,z^2$ belonging to the same $n$-cylinder we have
                    \begin{align*}
                        \left|S_n\psi(z^1)-S_n\psi(z^2)\right|\leq 2\sum_{k=1}^{\lceil\frac{n}{2}\rceil}\text{var}_k\psi
                    \end{align*}
                    From this we can say that there exists $N_2$ such that for all $n\geq N_2$ we have 
                    \begin{align*}
                        \tfrac{1}{n}\left|S_n\psi(z^1)-S_n\psi(z^2)\right|\leq \tfrac{\epsilon}{6}
                    \end{align*}
                    
                    \noindent Let
                    \begin{align*}
                        G_2^n=\left\{z':\left|\tfrac{S_n\psi(z')}{n}-\int\psi\,d\eta\right|<\tfrac{\epsilon}{3} \text{ for all } z'\in \cylz{(z)_0^{n-1}}\right\}.
                    \end{align*}
                    Note that $G_2^n$ is a union of $n$-cylinder sets.\\
                    
                    \noindent \textbf{Claim:} For $n\geq \max\{N_1,N_2\}$ we have $G_1^n\subseteq G_2^n$.\\
                    
                    \noindent Let $z\in G_1^n$ and let $z'\in \cylz{(z)_0^{n-1}}$. By the triangle inequality we have,
                    \begin{align*}
                        &\left|\tfrac{S_n\psi(z')}{n}-\int\psi\,d\eta\right|\\
                        &\leq
                         \left|\tfrac{S_n\psi(z)}{n}-\int\psi\,d\eta\right|+\tfrac{1}{n}\left|S_n\psi(z)-S_n\psi(z')\right|\\
                         &\leq \tfrac{\epsilon}{6}+\tfrac{\epsilon}{6}=\tfrac{\epsilon}{3}
                    \end{align*}
                    This means $\eta(G_2^n)\geq \eta(G_1^n)>\frac{7}{8}$.\\

                    \noindent From the Theorem \ref{wsmb} we can further say that there exists $N_3$ such that for all $n\geq N_3$ we have $\eta(\{z: \eta(\cylz{(z)_0^{n-1}})\in [e^{-n(h_\eta+\epsilon)},e^{-n(h_\eta-\epsilon)}]\})>\frac{7}{8}$.\\

                    \noindent Let 
                    \begin{align*}
                        G_3^n= \{z: \eta(\cylz{(z)_0^{n-1}})\in [e^{-n(h_\eta+\epsilon)},e^{-n(h_\eta-\epsilon)}]\}
                    \end{align*}

                    \noindent Let $N_4\in \mathbb N$ be such that  $\frac{1}{n}\log4<\frac{\epsilon}{3}$ for all $n\geq N_4$.\\
                    
                    \noindent Now let $n\geq \max\{N_1,N_2,N_3, N_4\}$. Let $\{C_1,\ldots,C_m\}$  be a collection of $n$- cylinders that has $\eta$- measure of the union greater than $\frac{1}{2}$. Note that $G_2^n, G_3^n$ are both union of $n$-cylinders. Let $G_4^n=\bigcup_{i=1}^m C_i$. Since $\eta(G_4^n)>\frac{1}{2},\eta(G_2^n)>\frac{7}{8}$ and $\eta(G_3^n)>\frac{7}{8}$ there exists a sub-collection $\{C_{i_1},\dots,C_{i_{m_1}}\}$ of  $\{C_{1},\dots,C_{{m}}\}$ that has $\eta$- measure of the union  greater than $\frac{1}{4}$ such that $C_{i_j}\subseteq G_k^n$ for all $j\in \{1,\ldots,m_1\}$ and $k\in\{2,3\}$.\\
                    
                    \noindent Now we can find a lower bound for $m_1$ as follows,
                    \begin{align*}
                        &\sum_{l=1}^{m_1} \eta(C_{k_l})\geq \tfrac{1}{4}\\
                        &\implies m_1 e^{-n(h_\eta-\frac{\epsilon}{3})}\geq \tfrac{1}{4}\\
                        & \implies m_1\geq \tfrac{1}{4} e^{n(h_\eta-\frac{\epsilon}{3})}. 
                    \end{align*}
                    Using the lower bound on $m_1$ and from the weak Birkhoff Theorem we get,
                    \begin{align*}
                        &\tfrac{1}{n}\log\left(\sum_{j=1}^m e^{S_n\psi(z_j)}\right)\\
                        &\geq \tfrac{1}{n}\log\left(\sum_{l=1}^{m_1} e^{S_n\psi(z_{k_l})}\right)\\
                        & \geq \tfrac{1}{n}\log(m_1e^{n(\int\psi\,d\eta-\tfrac{\epsilon}{3})})\\
                        &\geq \tfrac{1}{n}\log(\tfrac{1}{4} e^{n(h_\eta+\int\psi\,d\eta-\frac{2\epsilon}{3})})\\
                        &\geq \tfrac{1}{n}\log( e^{n(h_\eta+\int\psi\,d\eta-\epsilon)})\\
                        &= h_\eta+\int\psi\,d\eta-\epsilon
                    \end{align*}
                \end{proof}
                \begin{lemma}\label{lemma2}
                        Let $\eta$ be an ergodic measure on $Z$ such that $\eta(Y)=0$. Let $w\notin \mathcal{L}(Y)$ be a word such that $\eta(\cylz{w})>0$. Let $\gamma=\frac{\eta(\cylz{w})}{2}$. There exists $N\in \mathbb N$ such that for all $n\geq N$ there exists a collection of words $W_1,\ldots,W_m\in \mathcal{L}_n(Z)$ with
                        \begin{align*}
                            \sum_{j=1}^m\eta(\cylz{W_j})>\tfrac{1}{2}
                        \end{align*}
                        and 
                        \begin{align*}
                        \sum_{i=0}^{n-|w|}\textbf{1}_{\cylz{w}}(\sigma^i(z_j))>n\gamma
                    \end{align*} for any $z_j\in \cylz{W_j}$.
                        
                    \end{lemma}
                    \begin{proof}
                    
                    By Theorem \ref{wb} there exists $N\in \mathbb N$ such that
                    \begin{align*}
                        &\eta\left(\left\{z:\tfrac{1}{n}\sum_{i=0}^{n-|w|}\textbf{1}_{\cylz{w}}(\sigma^i(z))>\gamma\right\}\right) >\tfrac{1}{2}.
                    \end{align*}
                    and $|w|$ is the length of $w$.\\
                    
                    \noindent This set is a union of $n$-cylinders $C_1,\ldots C_m$ which we write as $\cylz{W_1},\ldots,\cylz{W_m}$.\\
                    
                    \noindent We have
                    \begin{align*}
                        \sum_{j=1}^m\eta\left(\cylz{W_j}\right)>\frac{1}{2}
                    \end{align*}
                    and
                    \begin{align*}
                        \frac{1}{n}\sum_{i=0}^{n-|w|}\textbf{1}_{\cylz{w}}(\sigma^i(z))>\gamma
                    \end{align*}
                    for any $z\in \cylz{W_j}$ and for any $1\leq j\leq m$.
                    
                    \end{proof}

                    \begin{theorem}\label{one}
                        $\mu$ is the unique equilibrium state of $\psi$. 
                    \end{theorem}
                    \begin{proof}
                        Let $\eta$ be an ergodic equilibrium state of $\psi$ that is not supported on $Y$. So there exists $w\notin \mathcal{L}(Y)$ such that $\eta(w)>0$. Let $|w|$ be the length of $w$. Let $\epsilon>0$ . By Lemma \ref{lemma1} there exists $N_1$ such that for all $n\geq N_1$ and $n$-cylinders  $C_1,\ldots C_m$ such that $\sum_{j=1}^m\eta (C_j)>\frac{1}{2}$ we have \[\tfrac{1}{n}\log\left(\sum_{j=1}^m e^{S_n\psi(z_j)}\right)>h_\eta+\int\psi\,d\eta-\tfrac{\epsilon}{2}\] where $z_j\in \cylz{C_j}$. By Lemma \ref{lemma2} we know that for $\gamma=\frac{\eta(\cylz{w})}{2}$ there exists $N_2\in \mathbb N$ such that for all $n\geq N_2$ there exists words $W_1,\ldots, W_m\in \mathcal{ L}_n(Y)$ such that 
                        \begin{align}\label{uni1}
                            \sum_{j=1}^m\eta(\cylz{W_j})>\tfrac{1}{2}.
                        \end{align}
                        and 
                        \begin{align}\label{uni2}
                        \sum_{i=0}^{n-|w|}\textbf{1}_{\cylz{w}}(\sigma^i(z_j))>n\gamma
                    \end{align} for any $z_j\in \cylz{W_j}$.\\
                    
                    \noindent Let $N_3$ such that for all $n\geq N_3$ we have
                    \begin{align}\label{uni3}
                        \tfrac{\log2K}{n}<\tfrac{\epsilon}{2}.
                    \end{align}
                    
                    \noindent Let $n\geq \max(N_1,N_2, N_3)$. Let $W_1,\ldots,W_m$ be as above and $z_j\in \cylz{W_j}$ for all $1\leq j\leq m$.  We have \[\frac{1}{n}\log\left(\sum_{j=1}^m e^{S_n\psi(z_j)}\right)>h_\eta+\int\psi\,d\eta-\frac{\epsilon}{2}.\]\\
                    
                    \noindent From (\ref{uni2}) we know that $W_j$ has at least $n\gamma$ many complete $w'$s. This means there are $\frac{n\gamma}{|w|}$ many non overlapping $w'$s in $W_j$. Let $l_j$ be the number of maximal words in $\mathcal{L}(Y)$ that $W_j$ is broken into, i.e. we have the following,
                    \begin{align*}
                        W_j=\omega_0\ldots\omega_{q_1-1}|\omega_{q_1}\ldots\omega_{q_2-1}|\ldots|\omega_{q_{l_j}}\ldots\omega_{q_{l_j-1}}
                    \end{align*}
                    where $\omega_{q_k}\ldots\omega_{q_{k+1}-1}$ are the maximal words  in $\mathcal{L}(Y)$ that $W_j$ is broken into. Note that these maximal segments cannot have $(\omega)_{p}^{p+|w|-1}=w$ for some $q_k\leq p<p+|w|-1<q_{k+1}-1$. That means for each of the non overlapping $w'$s we must have a split in $W_j$ that separates two maximal words. This shows that $l_j\geq \frac{n\gamma}{|w|}$, i.e. the number of maximal pieces in $\mathcal{L}(Y)$ that $(z_j)_0^{n-1}$ can be broken into is at least $\frac{n\gamma}{|w|}$ for all $1\leq j\leq m$. By using the estimates (\ref{est1}) and (\ref{est2}) we have 
                    \begin{align*}
                        &\sum_{j=1}^m e^{S_n\psi(z_j)}\\
                        &\leq K\sum_{l\geq \frac{n\gamma}{|w|}}\left(\sum_{n_1+..+n_l=n}\,\prod_{i=1}^{l}\left(\sum_{W_i\in \mathcal{L}_{n_i}(Y)} \left(e^{ S_{n_i}\psi_0(y_{W_i})-12\sqrt{n_i}}\right)\right)\right)\\
                        &= K \sum_{l\geq \frac{n\gamma}{|w|}} Q_n^l.
                    \end{align*}
                    Therefore
                    \begin{align*}
                        \log\left(\sum_{j=1}^m e^{S_n\psi(z_j)}\right)\leq \log K+\log\left(\sum_{l\geq \frac{n\gamma}{|w|}} Q_n^l\right)
                    \end{align*} From Lemma \ref{exists3} we know that
                    \begin{align*}
                    Q_n^l\leq \left(\frac{1}{2}\right)^le^{-2\mu(\cylz{\one})n}
                \end{align*}
                Combining the above inequalities we get
                \begin{align*}
                    &\log\left(\sum_{j=1}^m e^{S_n\psi(z_j)}\right)\\
                    &\leq \log K -2\mu(\cylz{\one})n+\log\left(\sum_{l\geq \frac{n\gamma}{|w|}} \left(\frac{1}{2}\right)^le^{-2\mu(\cylz{\one})n}\right)\\
                    &=\log 2K-\frac{n\gamma}{|w|}\log2-2\mu(\cylz{\one})n
                \end{align*}
                So we have 
                \begin{align*}
                    &h_\eta+\int\psi\,d\eta\\
                    &\leq -2\mu(\cylz{\one})- \frac{\gamma}{|w|}\log2+\frac{\log2K}{n}+\frac{\epsilon}{2}\\
                    &\leq -2\mu(\cylz{\one})-\frac{\gamma}{|w|}\log2+\epsilon.
                \end{align*}
                Since $P(\psi)=-2\mu(\cylz{\one})$ and $\epsilon$ is arbitrary  we have 
                \begin{align}
                    h_\eta+\int\psi\,d\eta<P(\psi)
                \end{align}
                This shows that $\eta$ is not an equilibrium state of $\psi$ which contradicts the assumption. So $\mu$ is the unique equilibrium state.

                    \end{proof}
                \noindent {Now we will prove Theorem \ref{thm2}.}
                
                    \begin{proof}

                        Let $n\geq 5$ and $W_n$ be the word  in $\mathcal{L}_n(Z)$ containing the point $o$ of all $\one$'s. We know that
                        $$\mu(\cylz{W_n})=\nu\left(\bigcap_{i=0}^{n-1}T^{-i}A\right).$$
                        
                        \noindent Also we have $\cylx{0^n}\subseteq T^{-i}A$ for all $0\leq i<n$ which gives us the following lower bound on  $\mu(\cylz{W_n})$ 
                        
                        $$\mu(\cylz{W_n})=\nu\left(\bigcap_{i=0}^{n-1}T^{-i}A\right)>\frac{1}{2^n}$$
                        \noindent Now note that $-nP(\psi)+S_n\psi(o)=2{\mu(\cylz{\one})n}-2n=-2\mu(\cylz{\zero})n$. By Lemma \ref{calc} we know that $ \mu(\cylz{\one})\leq \frac{6}{32}$ so we have $\mu(\cylz{\zero})\geq\frac{26}{32}$. Therefore,
                        \begin{align}\label{contra}
                            \frac{\mu(\cylz{W_n})}{e^{-nP(\psi)+S_n\psi(o)}}>\left(\frac{e}{2}\right)^n.
                        \end{align}
                        Since $\frac{\mu(\cylz{W_n})}{e^{-nP(\psi)+S_n\psi(o)}}$ grows faster than $\left(\frac{e}{2}\right)^n$ this shows that the very weak Gibbs inequality is violated at the point $o$ of all $\one'$s.
                    \end{proof}
                    \section{Proof of Theorem \ref{thm1}}\label{a.e.}

\noindent For the proof of Theorem \ref{thm1} we will recall the Variational Principle, the Shannon-McMillan-Breiman Theorem and the Birkhoff Theorem.
 \begin{proof}
 
Let $\mu$ be an ergodic equilibrium  state of a continuous potential $\psi$. Let $\epsilon>0$ and let $\mathcal{P}$ be a generating partition.\\

\noindent By the Shannon-McMillan-Breiman Theorem and the Birkhoff Theorem there exists a set $\mathcal{B}$ with $\mu(\mathcal{B})=1$ such that for $z\in \mathcal{B}$ we have,
     \begin{align*}
         \lim_{n\to \infty}\frac{\log(\mu(\mathcal{P}_n(z)))}{n}=-h(T).
     \end{align*}
     and 
     \begin{align*}
     \lim_{n\to \infty}\frac{S_n\psi(z)}{n}= \int \psi d\mu
 \end{align*}
 
\noindent So there exists $N(z,\epsilon)$ such that for all $n\geq N(z,\epsilon)$ we have,

\begin{align}\label{insmb}
    \left|-\frac{\log(\mu(\mathcal{P}_n(z)))}{n}- h(T)\right|<\frac{\epsilon}{2}
\end{align}
and 

\begin{equation}\label{inb}
    \left|\frac{S_n\psi(z)}{n}-\int \psi d\mu \right|<\frac{\epsilon}{2}
\end{equation}
From (\ref{insmb}) and (\ref{inb}) we get
\begin{align}\label{a.e.1}
    -\frac{\epsilon}{2}<-\frac{\log(\mu(\mathcal{P}_n(z)))}{n}- h(T)<\frac{\epsilon}{2}
\end{align}
and 
\begin{equation}\label{a.e.2}
    -\frac{\epsilon}{2}<\frac{S_n\psi(z)}{n}-\int \psi d\mu <\frac{\epsilon}{2}
\end{equation}
Adding (\ref{a.e.1}) and (\ref{a.e.2}) we have
\begin{equation}
    -\epsilon<\frac{S_n\psi(z)}{n}-\frac{\log(\mu(\mathcal{P}_n(z)))}{n}- h(T)-\int \psi d\mu<\epsilon
\end{equation}
Since $\mu$ is an equilibrium state, using the variational principle we have,
\begin{equation}
    -\epsilon<\frac{S_n\psi(z)}{n}-\frac{\log(\mu(\mathcal{P}_n(z)))}{n}- P(\psi)<\epsilon
\end{equation}

\noindent giving
\begin{equation}
    \frac{1}{e^{n\epsilon}}<\frac{\mu(\mathcal{P}_n(x))}{e^{S_n\psi(x)-n P_\phi}}<e^{n\epsilon}
\end{equation}
for all $n\geq N(z,\epsilon)$.\\

\noindent Hence for all $\epsilon>0$, $n\in \mathbb N$ and $\mu$ a.e. $z$, there exists $C(\epsilon,z)$ such that
\begin{align*}
    \frac{1}{C(z,\epsilon)e^{n\epsilon}}<\frac{\mu(\mathcal{P}_n(z))}{e^{S_n\phi(x)-n P_\phi}}<C(z,\epsilon)e^{n\epsilon}.
\end{align*}

 \end{proof}

\newpage

                    \bibliographystyle{plain}
                    \bibliography{bibliography}
\end{document}